\documentclass{IEEEtran}

\usepackage{xspace,amssymb,epsfig,syntonly}
\usepackage{epsfig,amsmath,color}
\usepackage{xspace,syntonly,empheq}
\usepackage{url}
\usepackage{bbm}
\usepackage{mathtools}
\usepackage{cite}
\usepackage{graphicx}
\usepackage{algorithm,algorithmic}
\usepackage{verbatim}
\usepackage{amssymb}
\usepackage{amsthm}
\usepackage{enumerate}

\newcommand{\utwi}[1]{\mbox{\boldmath $#1$}}

\newcommand{\trace}{{\textrm{Tr}}}

\newcommand{\diag}{\mathsf{diag}}

\newcommand{\cD}{{\cal D}}

\newcommand{\cA}{{\cal A}}

\newcommand{\cB}{{\cal B}}
\newcommand{\cU}{{\cal U}}

\newcommand{\cY}{{\cal Y}}

\newcommand{\ba}{{\bf a}}

\newcommand{\bd}{{\bf d}}
\newcommand{\be}{{\bf e}}

\newcommand{\bg}{{\bf g}}

\newcommand{\bp}{{\bf p}}
\newcommand{\bq}{{\bf q}}

\newcommand{\bx}{{\bf x}}

\newcommand{\bv}{{\bf v}}

\newcommand{\bz}{{\bf z}}
\newcommand{\by}{{\bf y}}
\newcommand{\bA}{{\bf A}}
\newcommand{\bB}{{\bf B}}

\newcommand{\bD}{{\bf D}}

\newcommand{\bG}{{\bf G}}

\newcommand{\bH}{{\bf H}}

\newcommand{\bM}{{\bf M}}

\newcommand{\bR}{{\bf R}}

\newcommand{\bI}{{\bf I}}

\newcommand{\bDelta}{{\utwi{\Delta}}}

\newcommand{\blambda}{{\utwi{\lambda}}}

\newcommand{\bmu}{{\utwi{\mu}}}

\newcommand{\proj}{\mathrm{proj}}

\newcommand{\sfH}{\textsf{H}}
\newcommand{\sfT}{\textsf{T}}

\DeclareMathOperator*{\argmin}{arg\,min}

\DeclarePairedDelimiterX{\norm}[1]{\lVert}{\rVert}{#1}

\usepackage{epstopdf}

\begin{document}

\newtheorem{definition}{Definition}
\newtheorem{assumption}{Assumption}
\newtheorem{proposition}{Proposition}
\newtheorem{theorem}{Theorem}
\newtheorem{lemma}{Lemma}
\theoremstyle{definition}
\newtheorem{remark}{Remark}

\def\HS{\hspace{\fontdimen2\font}}
\font\myfont=cmr12 at 16pt

\IEEEoverridecommandlockouts


\title{Network-Cognizant Voltage Droop Control \\for Distribution Grids} 

\author{Kyri Baker, Andrey Bernstein, Emiliano Dall'Anese, and Changhong Zhao
\thanks{K. Baker is with the College of Engineering and Applied Science at the University of Colorado, Boulder. Email: kyri.baker@colorado.edu}
\thanks{A. Bernstein, E. Dall'Anese, and C. Zhao are  with the National Renewable Energy Laboratory (NREL), Golden, CO, USA. Emails: name.lastname@nrel.gov}
\thanks{This work was supported by the Laboratory Directed Research and Development Program at NREL and by funding from the Advanced Distribution Management Systems Program of the U.S. Department of Energy's Office of Electricity Delivery and Energy Reliability under Lawrence Berkeley National Laboratory Contract No. DE-AC02-05CH11231.}
}

\maketitle
\vspace{-1cm}
\begin{abstract}
This paper examines distribution systems with a high integration of distributed energy resources (DERs)  and addresses the design of local control methods for real-time voltage regulation. Particularly, the paper focuses on proportional control strategies where the active and reactive output-powers of DERs are adjusted in response to (and proportionally to) local changes in voltage levels. The design of the voltage-active power and voltage-reactive power characteristics leverages suitable linear approximations of the AC power-flow equations and is network-cognizant; that is, the  coefficients of the controllers embed information on the  location of the DERs and forecasted non-controllable loads/injections and, consequently, on the effect of DER power adjustments on the overall voltage profile. A robust approach is pursued to cope  with uncertainty in the forecasted non-controllable loads/power injections. Stability of the proposed local controllers is analytically assessed and numerically corroborated. 
\end{abstract}


\section{Introduction}
\label{sec:Introduction}

The increased deployment of renewable energy resources such as photovoltaic (PV) systems operating with business-as-usual practices has already precipitated a unique set of power-quality and reliability-related concerns at the distribution-system level~\cite{Liu08,Woyte06}. For example, in settings with high renewable penetration, reverse power flows increase the likelihood of voltages violating prescribed limits (e.g,. ANSI C84.1 limits). Furthermore, volatility of ambient conditions leads to rapid variations in renewable generation and, in turn, to increased cycling and wear-out of legacy voltage regulation equipment. 

To alleviate these concerns, some recent research efforts focused on development of local (i.e., autonomous) control strategies where each  power-electronics-interfaced distributed energy resource (DER) adjusts its output powers based on voltage measurements at the point of connection~\cite{Hoke15, Jahangiri, BakerISGT17, Farivar15, Xinyang15, Zhu16, Xinyang16}. Particularly, inverter-interfaced DERs implementing the so-called \emph{Volt/VAR control} or \emph{voltage droop control} have been shown to effectively aid voltage regulation by absorbing or providing reactive power in response to (and proportionally to) local changes in voltage magnitudes. 

Focusing on Volt/VAR control, a recommended setting for the voltage-reactive power characteristic for DERs is specified in the IEEE 1547.8 Standard \cite{IEEE1547}. However, the design of the voltage-reactive power characteristics is network-agnostic -- in the sense that it does not take into account the location of the DERs in the feeder and, thus, the effect of output power adjustments on the overall voltage profile. Further, the Volt/VAR mechanism specified in~\cite{IEEE1547} may exhibit oscillatory behaviors and its stability (in an input-to-state stability sense) is still under investigation~\cite{Hoke15, BakerISGT17, Xinyang15, Xinyang16}.

Several works addressed the design of local Volt/VAR controllers for voltage regulation purposes. For example, \cite{Farivar15, Xinyang15, Zhu16, Xinyang16} synthesized Volt/VAR controllers by leveraging  optimization and game-theoretic arguments. Stability claims were derived based on a linearized AC power-flow model in~\cite{Farivar15, Xinyang15, Zhu16}, while~\cite{Xinyang16} analyzed the stability of  incremental Volt/VAR controllers in the purview of the nonlinear AC power-flow equations. On the other hand, heuristics were utilized in, e.g.,~\cite{Calderaro14}. Approaches based on extremum-seeking control have been explored as a model-free alternative to Volt/VAR control~\cite{Arnold16}; however, it may be difficult to systematically take into account the network effects to design the control rule, especially in meshed and unbalanced systems. Capitalizing on the fact that distribution networks typically exhibit a high resistance-to-reactance ratio, additional works considered active power control to possibly improve efficiency and better cope with  overvoltage conditions \cite{Morren05,Samadi14,Tina15}.  Active and reactive power control in microgrids was studied in \cite{Lee17, Maleki14}.

However,\cite{Farivar15, Xinyang15, Zhu16, Xinyang16, Morren05} (and pertinent references therein) do not address the design of the voltage-reactive power and voltage-active power characteristics; rather, for given coefficients of the controllers (i.e., droop coefficients), rules to update the active power and/or reactive power are designed with the objective of ensuring a stable system operation. In addition to assuming the droop coefficients are given, \cite{Farivar15, Xinyang15, Zhu16, Xinyang16} employ an  incremental update strategy  to ensure stability. On the other hand,~\cite{Calderaro14, Samadi14, Lee17, Maleki14} addressed the problem of computing the coefficients of the controllers; however, the designs are based on heuristics and, hence, no optimality or stability claims are provided. 

This paper addresses the design of proportional control strategies wherein  active and/or reactive output-powers of DERs are adjusted in response to  local changes in voltage levels -- a methodology that we occasionally refer to as \emph{Volt/VAR/Watt control}. 

The voltage-active power and voltage-reactive power characteristics are  obtained  based on the following design principles:

i) Suitable linear approximations of the AC power-flow equations \cite{christ2013sens,sairaj2015linear,swaroop2015linear,bolognani2015linear} are utilized to render the voltage-power characteristics of individual DERs \emph{network-cognizant}; that is, the  coefficients of the controllers embed information about the  location of the DERs and non-controllable loads/injections and, consequently, on the effect of DER power adjustments on the overall voltage profile (rather than just the  effect on the voltage at the point of interconnection of the DER). 

ii)  A \emph{robust} design approach is pursued to cope  with uncertainty in the forecasted non-controllable loads/power injections. 

iii) The controllers are obtained with the objective of ensuring a stable system operation, within a well-defined notion of input-to-state stability.

Based on the design guidelines i)--iii) above,  the coefficients of the proportional controllers are obtained by solving a robust optimization problem. The optimization problem is solved at regular time intervals (e.g., every few minutes) so that the droop coefficients can be adapted to new operational conditions. The optimization problem  can accommodate a variety of performance objectives such as minimizing voltage deviations from a given profile, maximizing stability margins, and individual consumer objectives (e.g., maximizing active power production). By utilizing sparsity-promoting regularization functions \cite{Tibshirani94}, the proposed approach also enables selection of subsets of locations where Volt/VAR/Watt control is critical to ensure voltage control~\cite{OID}. 
The proposed framework subsumes existing  Volt/VAR control by simply forcing the Volt/Watt coefficients to zero in the optimization problem.
 
The paper is organized as follows: Section \ref{sec:models_and_approx} describes the model of the distribution grid and the AC power flow linearization; Section \ref{sec:control} presents our approach and formulates the optimization problem used to design the controllers; Section \ref{sec:robust} introduces its robust counterpart; Section \ref{sec:variants} presents a few possible objectives that could be considered in the controller design; Section \ref{sec:sims} provides a numerical analysis performed using the IEEE 37-node test feeder, including a sensitivity analysis of the proposed design framework and implementation of both single-phase  and a three-phase unbalanced systems; and lastly, Section \ref{sec:conclusion} concludes the paper.

\section{System model}
\label{sec:models_and_approx}

Consider a distribution system\footnote{Upper-case (lower-case) boldface letters will be used for matrices (column vectors); $(\cdot)^\sfT$ for transposition; 
  and $|\cdot|$ denotes the absolute value of a number or the cardinality of a set. Let $\cA \times \cB$ denote the Cartesian product of sets $\cA$ and $\cB$.  
  For a given $N \times 1$ vector $\bx \in \mathbb{R}^N$, $\|\bx\|_2 := \sqrt{\bx^\sfH \bx}$; $\|\bx\|_\infty := \max(|x_1|...|x_n|)$; and $\diag(\bx)$ returns a $N \times N$ matrix with the elements of $\bx$ in its diagonal. 
  The spectral radius $\rho(\cdot)$ is defined for an $N\times N$ matrix $\bA$ and corresponding eigenvalues $\lambda_1 ... \lambda_N$ as $\rho(\bA) := \max(|\lambda_1|, ... ,|\lambda_N|)$. For an $M\times N$ matrix $\bA$, the Frobenius norm is defined as $||\bA||_F = \sqrt{\trace(A^*A)}$ and the spectral norm is defined as $||\bA||_2 := \sqrt{\lambda_{max}(\bA^*\bA)}$, where $\lambda_{max}$ denotes maximum eigenvalue. Finally, $\bI_N$ denotes the $N \times N$ identity matrix. } comprising  $N+1$ nodes collected in the set $\mathcal{N}  \cup \{0\}$, $\mathcal{N} := \{1, \ldots, N\}$. Node $0$ is defined to be the distribution substation. Let $v_n$ denote the voltage at node $n = 1,\ldots N$ and let  $\bv :=  [|v_1|, \ldots, |v_N|]^\sfT \in \mathbb{R}^{N}$ denote the vector collecting the voltage magnitudes.

Under certain conditions, the non-linear AC power-flow equations can be compactly written as
\begin{align} 
\bv = F(\bp, \bq), \label{eq:pf_exact} 
\end{align}
where $\bp \in \mathbb{R}^N$ and $\bq \in \mathbb{R}^N$ are vectors collecting the net active and reactive power injections, respectively, at  nodes $n=1...N$. The existence of the power-flow function $F$ is related to the question of existence and uniqueness of the power-flow solution and was established in several recent papers under different conditions\footnote{In this paper, $F$ is used only to analyze the stability of the proposed controllers, and thus \eqref{eq:pf_exact} can be considered as a ``black box'' representing the reaction of the power system to the net active and reactive power injections $(\bp, \bq)$. In fact, this view does not require uniqueness of the power-flow solution by allowing the function $F$ to be time-dependent.} \cite{congLF,Bolognani15}.

Nonlinearity of the AC power-flow equations poses significant challenges 
with regards to solving problems such as optimal power flow  as well as the design of the proposed decentralized control strategies for DERs. 
Thus, to facilitate the controllers' design, linear approximations of \eqref{eq:pf_exact} are utilized in this paper. In particular, we consider a linear relationship between voltage magnitudes and injected active and reactive powers of the following form: 
\begin{align} 
\bv & \approx F_L(\bp, \bq) = \bR \bp + \bB \bq + \ba. \label{eq:approximate} 
\end{align}
System-dependent matrices $\bR \in \mathbb{R}^{N \times N}$, $\bB \in \mathbb{R}^{N \times N}$, and vector $\ba \in \mathbb{R}^{N}$ can be computed in a variety of ways:

i) Utilizing suitable linearization methods for the AC power-flow equations, applicable when the network model is known; see e.g.,~\cite{Baran89,christ2013sens,sairaj2015linear,swaroop2015linear,bolognani2015linear,multiphaseArxiv,linModels} and pertinent references therein; and,  

ii) Using regression-based methods, based on real-time measurements of $\bv$, $\bp$, and $\bq$. E.g., the recursive least-squares method can be utilized to continuously update the model parameters. 

It is worth emphasizing that the linear model \eqref{eq:approximate} is  utilized to facilitate the  design of the optimal controllers; on the other hand,  the stability analysis and numerical experiments are performed using the exact (nonlinear) AC power-flow equations. Moreover, 
we note that the accuracy of the linear model is  dependent on the particular linearization method. For example, \cite{multiphaseArxiv} presents linear models that provide accurate representations of the voltages under variety of loading conditions.

\vspace{.1cm}

\begin{remark}
For notational and exposition simplicity, the proposed framework is outlined for a balanced distribution network. However, the proposed control framework is naturally applicable to multi-phase unbalanced systems with any topology. In fact, the linearized model~\eqref{eq:approximate} can be readily extended to the multi-phase unbalanced setup as shown in  e.g.,~\cite{christ2013sens,multiphaseArxiv}, and the controller design procedure outlined in the ensuing section can be utilized to compute the Volt/VAR/Watt characteristics of devices located at any bus and phase. 
To demonstrate this, we performed numerical experiments of a three-phase system in Section \ref{sec:sensitive}.
\end{remark}

\section{Controller Design}
\label{sec:control}

In this section, the main concept of the tuning the coefficients of the droop controllers for active and reactive power is discussed. Below is the outline of our approach:
\begin{itemize}
\item \textbf{Optimal droop controllers design.} On a slow time-scale (e.g., every 5-15 minutes), update the parameters of the linear model and forecasts of solar and load, and compute the coefficients of the droop controllers based on the knowledge of the network, with the objective of minimizing voltage deviations while keeping the system stable.  
\item \textbf{Real-time operation.} On a fast time-scale (e.g., subsecond), adjust active and reactive powers of DERs locally, based on the recently computed coefficients. Ensure the resulting adjustments are within the inverter operational constraints by projection onto the feasible set of operating points.
\end{itemize}

\subsection{Problem Formulation} \label{sec:initial}
Consider a discrete-time decision problem of adjusting active and reactive power setpoints during real-time operation in response to local changes in voltage magnitudes. Let $k = 1, 2, \ldots$ denote the time-step index, and let the 
voltage magnitudes at time step $k$ be expressed as 
\begin{equation} \label{eqn:exact_v}
\bv(k) = F(\bp(k) + \bDelta \bp(k), \bq(k) + \bDelta \bq(k)),
\end{equation}
where $\bp(k)$ and 
$\bq(k)$ are the active and reactive powers setpoints, respectively, throughout the feeder and $\bDelta \bp (k)$ and $\bDelta \bq (k)$ are the vectors of active and reactive power adjustments of the Volt/VAR/Watt controllers. Also, consider a given power-flow solution $\bar{\bv}, \bar{\bp}, $ and $  \bar{\bq}$ satisfying \eqref{eq:pf_exact} and \eqref{eq:approximate}; see e.g., \cite{christ2013sens,bolognani2015linear}. The triple $(\bar{\bv}, \bar{\bp}, \bar{\bq})$ can be viewed as a reference power-flow solution (e.g., a linearization point of \eqref{eq:approximate}). Finally, let $\bDelta \bv (k) := \bv(k) - \bar{\bv}$ denote the voltage deviation from $\bar{\bv}$. 

The objective is to design a decentralized proportional real-time controller to update $\bDelta \bp(k)$ and $\bDelta \bq(k)$ in response to $\bDelta \bv (k-1)$. That is, the \emph{candidate} adjustments are given by
\begin{equation} \label{eqn:controller}
\bDelta \tilde{\bp}(k) = \bG_p \bDelta \bv(k-1), \quad  \bDelta \tilde{\bq}(k) = \bG_q \bDelta \bv(k-1), 
\end{equation}
where $\bG_p$ and $\bG_q$ are \emph{diagonal} $N \times N$ matrices collecting the coefficients of the proportional controllers. The change in active power output at node $n$ in response to a change in voltage at node $n$ is then given by each on-diagonal element in $\bG_p$, $g_{p, n} := (\bG_{p})_{nn}$, $n = 1, \ldots, N$; and the change in reactive power output at node $n$ in response to a change in voltage at node $n$ is given by each on-diagonal element in $\bG_q$, $g_{q, n} := (\bG_{q})_{nn}$, $n = 1, \ldots, N$. 

However, due to inverter operational constraints, setting $\bDelta \bp(k) = \bDelta \tilde{\bp}(k)$ and $\bDelta \bq(k) = \bDelta \tilde{\bq}(k)$ might not be feasible. We next account for this by projecting the candidate setpoint onto the feasible set. 
To this end, let $\cY_n(k)$ be the set of feasible operating points for an inverter located at node $n$ at time step $k$. 
For example, for a PV inverter with rating $S_n$ and an available  power $P_{\textrm{av},n}(k)$, the set $\cY_n(k)$ is given by 
$$ \cY_n(k)  =  \left\{({P}_{n}, {Q}_{n} ) \hspace{-.1cm} :  0 \leq {P}_{n}  \leq  P_{\textrm{av},n}(k), {Q}_{n}^2  \leq  S_{n}^2 - {P}_{n}^2 \right\}. $$
Notice that, for PV inverters, the set $\cY_n(k)$ is convex, compact, and time-varying (it depends on the available power $P_{\textrm{av},n}(k)$).

From~\eqref{eqn:controller}, a new potential setpoint for inverter $n$ is generated as $\tilde{P}_n(k) := P_n(k) + g_{p, n} \Delta V_n (k-1)$,  and $\tilde{Q}_n(k) := Q_n(k) + g_{q, n} \Delta V_n (k-1)$. If $(\tilde{P}_n(k), \tilde{Q}_n(k)) \notin \cY_n(k)$, then a feasible setpoint is obtained as: 
\begin{equation}
(\hat{P}_n(k), \hat{Q}_n(k)) = \proj_{\cY_n(k)} \{ (\tilde{P}_n(k), \tilde{Q}_n(k)) \} \label{eqn:controller_proj}
\end{equation}
where $$\proj_{\cY} \{\bz\} := \argmin_{\by \in \cY} \|\by - \bz \|_2$$ 
denotes the projection of the vector $\bz$ onto the convex set $\cY$. For typical systems, such as PV or battery, the projection operation in~\eqref{eqn:controller_proj} can be computed in closed form (see, e.g., \cite{opfPursuit}). In general, the set $\cY_n(k)$ can be approximated by a polygon, and efficient numerical methods can be applied to compute the projection (as in, e.g.,  \cite{commelec1}).  

\begin{remark}
There are multiple
ways to perform the projection  onto the feasible operating region, depending on the metric used (Euclidean norm, infinity norm, etc).  
Moreover, the feasible  setpoint can also be chosen using heuristics (for example, by neglecting the contribution from active/reactive power entirely and projecting onto the reactive/active plane, respectively). However, by utilizing the projection operator with respect to the Euclidean norm as proposed in this work, we guarantee the stability properties established in Theorem 1 below.
\end{remark}

In the next section, we give conditions under which the proposed controllers are stable in a well-defined sense, while in Section \ref{sec:opt_contr}, we use these stability conditions to design  optimal control coefficients $\bG_p, \bG_q$. 

\subsection{Stability Analysis} \label{sec:stab}
We next analyze the input-to-state stability properties of the proposed controllers by making reference to a given linear model \eqref{eq:approximate}. The following assumption is made.
\begin{assumption} \label{asm:lin_mod}
The error between the linear model \eqref{eq:approximate} and the exact power-flow model \eqref{eq:pf_exact} is bounded, namely there exists $\delta < \infty$ such  that  $\|F(\bp, \bq) - F_L(\bp, \bq)\|_2 \leq \delta$ for all (feasible) $\bp$ and $\bq$.
\end{assumption}

For future developments, let $\bG := [\bG_p,\bG_q]^\sfT$ be a $2N\times N$ matrix composed of two stacked $N \times N$ diagonal matrices $\bG_p$ and $\bG_q$. Also, let $\bar{\bz} := [\bar{\bp}^T, ~ \bar{\bq}^T]^T$, and $\Delta\bp_{nc}(k) := \bp(k) - \bar{\bp}$ and $\Delta\bq_{nc}(k) := \bq(k) - \bar{\bq}$ denote the deviation of the uncontrollable powers at time step $k$ from the nominal value. Let the matrix $\bH$ and the vector $\bDelta \bz_{nc}(k)$ be defined as $\bH := [\bR, \,\bB]$ and $\bDelta \bz_{nc}(k) := [\bDelta \bp_{nc}(k)^\sfT, \bDelta \bq_{nc}(k)^\sfT]^\sfT$, where $(\bR, \bB)$ are the parameters of the linear model \eqref{eq:approximate}. Finally, let  $\bDelta \bz(k) := [\bDelta \bp(k)^\sfT, \bDelta \bq(k)^\sfT]^\sfT$ denote the controllable change in active and reactive power of each inverter.

Let $\cY(k) :=  \cY_1(k) \times \ldots \times \cY_N(k)$ be the aggregate compact convex set of feasible setpoints at time step $k$. Also, let
\begin{equation}
\cD(k) := \{\bDelta \bz: \, \bz(k) + \bDelta \bz \in \cY(k) \}
\end{equation}
denote the set of feasible Volt/VAR/Watt adjustments, where $\bz(k) = [\bp(k)^\sfT, \bq(k)^\sfT]^\sfT$ denotes the power setpoint at time step $k$ before the Volt/VAR/Watt adjustment.
It is easy to see that $\cD(k)$ is a convex set as well, and that the projected Volt/VAR/Watt controller \eqref{eqn:controller_proj} is equivalently defined by
\begin{equation} \label{eqn:controller_proj1} 
\bDelta \bz (k) = \proj_{\cD(k)} (\bG \bDelta \bv(k-1)).
\end{equation}
Recall that $\bar{\bv} = F(\bar{\bp}, \bar{\bq}) = F_L(\bar{\bp}, \bar{\bq})$. The dynamical system imposed by \eqref{eqn:exact_v}, \eqref{eqn:controller}, and \eqref{eqn:controller_proj1} is then given by
\begin{equation} \label{eqn:dyn_sys_exact}
\bDelta \bv (k) = F\left(\bz(k) + \proj_{\cD(k)} (\bG \bDelta \bv(k-1))\right) - F_L(\bar{\bp}, \bar{\bq})
\end{equation}

The following result provides us with a condition for stability of \eqref{eqn:dyn_sys_exact} in terms of the parameters of the linear model $\bH$ and the controller active/reactive power coefficients $\bG$.

\begin{theorem} \label{thm:stab}
Suppose that Assumption \ref{asm:lin_mod} holds. Also assume that $r := \|\bG \bH\|_2 < 1$ and that $\|\bDelta \bz_{nc}(k)\|_2 \leq C$ for all $k$. Then
\[
\limsup_{k \rightarrow \infty} \|\bDelta \bv(k)\|_2 \leq  \frac{\|\bH\|_2 C + (1 - r + \|\bG \|_2\|\bH\|_2) \delta}{1 - r}.
\]
\end{theorem}

We note that Theorem \ref{thm:stab} establishes bounded-input-bound-state (BIBS) stability. Indeed, it states that under the condition $\|\bG \bH\|_2 < 1$, the state variables $\bDelta \bv(k)$ remains bounded whenever the input sequence $\{\bDelta \bz_{nc}(k) = \bz(k) - \bar{\bz}\}$ is bounded. Also, observe that the result of Theorem \ref{thm:stab} does not depend on the particular linearization method, as long as it satisfies Assumption \ref{asm:lin_mod}.

The proof of Theorem \ref{thm:stab} can be found in the Appendix. Next, we discuss the design of the controllers.

\subsection{Optimal Controller Design} \label{sec:opt_contr}

In this section, we propose an optimal design of droop coefficients $\bG := [\bG_p,\bG_q]^\sfT$. The objective is to minimize voltage deviations while keeping the system stable by explicitly imposing the condition $\|\bG \bH\|_2 < 1$ of Theorem \ref{thm:stab}. 
We leverage the following two simplifications that render the resulting optimization problem tractable:
\begin{enumerate}[(i)]
\item We consider a linear power-flow model \eqref{eq:approximate} instead of the exact one \eqref{eq:pf_exact};
\item We ignore the projection in the controllers' update.
\end{enumerate}
Based on these two simplifications, we obtain the following \emph{linear} dynamical system for voltage deviations (cf.~the exact non-linear dynamical system \eqref{eqn:dyn_sys_exact}):
\begin{align} 
\bDelta \tilde{\bv}(k) &= \bH \bDelta \bz_{nc}(k) + \bH \bDelta \bz(k) \nonumber \\
&= \bH \bDelta \bz_{nc}(k) + \bH \bG \bDelta \tilde{\bv}(k-1). \label{eqn:dyn}
\end{align}
We note that under the condition $\|\bG \bH\|_2 < 1$ of Theorem \ref{thm:stab}, we have that the spectral radius\footnote{See, e.g., \cite[Theorem 1.3.20]{H90} for the proof of the fact that  $\rho(\bH \bG) = \rho(\bG \bH)$  for any two matrices $\bH$ and $\bG$ with appropriate dimensions.} $\rho(\bH \bG) = \rho(\bG \bH) \leq \|\bG \bH\|_2 < 1$. Thus, from standard analysis in control of discrete-time linear systems, the system \eqref{eqn:dyn} is stable as well; see, e.g., \cite{Freeman1965}.

To design the controllers, we assume that a \emph{forecast} $\bmu$ for $\bDelta \bz_{nc}(k)$ is available. In particular, in this paper we compute $\bmu$ from the history  by averaging over the interval between two consecutive droop coefficient adjustments. However, other forecasting methods could be considered as well. Thus, define the following modified dynamical system that employs $\bmu$:

\begin{equation} \label{eqn:dyn_e}
\be(k+1) =  \bH \bG \be(k) + \bH \bmu.
\end{equation}
Note that as  $\rho(\bH \bG)  < 1$, the system \eqref{eqn:dyn_e} converges to the unique solution of the fixed-point equation
\[
\be =  \bH \bG \be + \bH \bmu
\]
given by
\[
\be = (\bI - \bH \bG)^{-1}\bH \bmu.
\]
Moreover, if the forecast $\bmu$ is accurate enough, namely $\| \bDelta \bz_{nc}(k) - \bmu \|_2 \leq \varepsilon$ for some (small) constant $\varepsilon$ and all $k$, then using the method  of proof of Theorem \ref{thm:stab} it can  be shown that
\[
\limsup_{k \rightarrow \infty} \|\bDelta \tilde{\bv}(k) - \be\|_2 \leq \frac{K \varepsilon}{1 - \rho(\bH \bG)}
\]
for some constant $K < \infty$, implying that  minimizing $\be$ also asymptotically minimizes $\bDelta \tilde{\bv}(k)$.

Hence, our goal in general is to design a controller $\bG$ that solves the following optimization problem:
\begin{subequations} 
\begin{align} 
\textrm{(P0)}&  \inf_{\substack{\bG, \be}}  \hspace{.2cm} f(\be, \bG) \\
&\mathrm{subject\,to} \nonumber \\
& \hspace{1cm}  \be =  (\bI - \bH \bG)^{-1}\bH \bmu \label{eqn:equal_constr}\\
&  \hspace{1cm} \|\bG \bH\|_2 < 1 \label{eqn:inequal_constr}
 \end{align}
\end{subequations}

\noindent for some convex objective function $f(\be, \bG)$.
However, this problem cannot be practically solved mainly due to: (i) non-linear equality constraint \eqref{eqn:equal_constr} and (ii) the fact that \eqref{eqn:inequal_constr} defines an open set. To address problem (i), we use the first two terms of the Neuman series of a matrix \cite{H90}:

\begin{equation}
(\bI-\bH \bG)^{-1}\bH \bmu \approx (\bI + \bH \bG)\bH\bmu. \label{eq:dev_approx}
\end{equation}

\noindent The sensitivity of this approximation to changes in $\bG$ is discussed further in Section \ref{sec:sims}. To address problem (ii), the strict inequality \eqref{eqn:inequal_constr} can be converted to inequality and included in an optimization problem by including a stability margin $\epsilon \geq \epsilon_0$ such that

\begin{equation} \label{eqn:inequal_constr_closed}
\|\bG \bH\|_2 \leq 1 - \epsilon
\end{equation}

\noindent where $\epsilon_0 > 0$ is a desired  lower bound on the stability margin. Finally, to further simplify this constraint, we upper bound the induced $\ell_2$ matrix norm with the Frobenius norm. 

Thus, (P0) is reformulated as the following:

\begin{subequations} 
\begin{align} 
\textrm{(P1)}&  \min_{\substack{\bG, \be, \epsilon}}  \hspace{.2cm} f(\be, \bG, \epsilon) \\
&\mathrm{subject\,to} \nonumber \\
& \hspace{1cm}  \be =  (\bI + \bH \bG)\bH\bmu \label{eqn:equal_constr_mod}\\
& \hspace{1cm} ||\bG \bH||_F \leq 1 - \epsilon \label{eqn:inequal_constr_mod}, \, \, i = 1, ..., N \\
& \hspace{1cm} \epsilon_0 \leq \epsilon \leq 1 \label{eqn:epsilon_constr}\\
& \hspace{1cm} \bG \leq 0 \label{eqn:pos_constr},
 \end{align}
\end{subequations}

\noindent where \eqref{eqn:pos_constr} ensures that each of the resulting coefficients are non-positive. As a first formulation of (P1), we consider minimizing the voltage deviation while providing enough stability margin, by defining the following objective function
\begin{equation} \label{eq:dev_obj}
f(\be, \bG, \epsilon) = \|\be\|_\infty - \gamma \epsilon,
\end{equation}
where $\gamma \geq 0$ is a weight parameter which influences the choice of the size of the stability margin $\epsilon$. The infinity norm was chosen in order to minimize the worst case voltage deviation in the system. 

\section{Robust Design} \label{sec:robust}

The optimization problem formulated in the previous section assumes that a forecast $\bmu$ is available, and a certainty equivalence formulation is derived. However, predictions are uncertain, and designing the coefficients for a particular $\bmu$ may result in suboptimality. 
Thus, in this section, we assume that the uncontrollable variables $\{\bDelta \bz_{nc}(k)\}$ belong to a \emph{polyhedral uncertainty set} $\cU$ (e.g, prediction intervals), 
and formulate the robust counterpart of (P1), which results in a convex optimization program\footnote{In practice, the set $\cU$ can be provided by prediction/forecasting tools. Hence, the detailed discussion of this choice is out of the scope of this paper.}. A robust design is well-justified  in distribution settings with high penetration of renewable sources of energy where forecasts of the available powers might be affected by large errors (e.g., in situations where solar irradiance is highly volatile).


In the spirit of~\eqref{eq:dev_approx}, we start the design by leveraging a truncated version of the Neuman series. To that end, we use the exact expression for $\bDelta \tilde{\bv} (k)$ obtained by applying \eqref{eqn:dyn} recursively:
\begin{align}
\bDelta \tilde{\bv}(k) &= \bH \left(\sum_{i = 0}^{k-1} (\bG \bH)^i \bz_{nc}(k-i) \right) \\
&= \bH \bDelta \bz_{nc}(k) + \bH \bG \bH \bDelta \bz_{nc}(k-1) + O\left( (\bG \bH)^2 \right). \nonumber
\end{align}

We next make the following two approximations:
\begin{enumerate}[(i)]
\item We neglect the terms $O\left( (\bG \bH)^2 \right)$. This is justified similarly to the Neuman series approximation \eqref{eq:dev_approx} under the condition that $\rho(\bG \bH) < 1$. 
\item We assume that the controllers are fast enough so that the variability of the uncontrollable variables in two consecutive Volt/VAR/Watt adjustment steps is negligible. Namely, we assume that $\bDelta \bz_{nc}(k) \approx \bDelta \bz_{nc}(k-1)$.
\end{enumerate}
Thus, $\bDelta \tilde{\bv}(k)$ is approximated as
\begin{equation} \label{eqn:dev_approx_new}
(\bI + \bH \bG)\bH  \bmu
\end{equation}
for some $\bmu \in \cU$; cf.~\eqref{eq:dev_approx}.

We next proceed to define a robust optimization problem that minimizes the $\ell_\infty$ norm of \eqref{eqn:dev_approx_new} for the worst-case realization of $\bmu \in \cU$.
Define $\bA(\bG) = (\bI + \bH \bG) \bH$
and rewrite the problem in epigraph form so that the uncertainty is no longer in the objective function: 
\begin{subequations} 
\begin{align} 
\textrm{(P2)}&  \min_{\substack{\bG, \epsilon, t}} \hspace{0.8cm}  t - \gamma \epsilon \\
&\mathrm{subject\,to} \nonumber \\
&  \hspace{1cm} \max_{\substack\bmu \in \cU}||\bA(\bG)\bmu||_{\infty} \leq t \label{subproblem}\\
& \hspace{1cm} \eqref{eqn:inequal_constr_mod}, \eqref{eqn:epsilon_constr}, \eqref{eqn:pos_constr} \nonumber
\end{align}
\end{subequations}

\noindent where $\cU = \{\bmu : \bD \bmu \leq \bd \}$ for matrix $\bD$ and vector $\bd$ of appropriate dimensions. The constraint \eqref{subproblem} can equivalently be written as the following set of constraints:
\begin{align}
\max_{\substack\bmu \in \cU}\bigg|\sum_{j=1}^n \bA_{i,j}(\bG) \mu_j\bigg| \leq t, ~~\forall i = 1 ...n
\end{align}

Splitting the absolute value into two separate optimization problems, we obtain the following constraints:
\begin{subequations}
\begin{align}
\bigg( \max_{\substack\bmu \in \cU}~~\sum_{j=1}^n \bA_{i,j}(\bG) \mu_j \bigg)  \leq t, ~~\forall i = 1 ...n \label{sub1}\\
\bigg( \max_{\substack\bmu \in \cU} ~-\sum_{j=1}^n \bA_{i,j}(\bG) \mu_j \bigg) \leq t, ~~\forall i = 1 ...n\label{sub2}
\end{align}
\end{subequations}

\noindent To formulate the final convex robust counterpart of (P1), the dual problems of \eqref{sub1} and \eqref{sub2} are sought (see, e.g., \cite{Bertsimas11}). For clarity, define $\ba_i^T$ as the $i$th row of $\bA$. Since $\bG$ is not an optimization variable in the inner maximization problems, the dual problems for \eqref{sub1} and \eqref{sub2} can be written as follows: 

\vspace{2mm}
\noindent \textbf{Dual problem of \eqref{sub1}:}
\begin{align*}
\max_{\substack\bmu}&~~\ba_i^T \bmu \hspace{1.4cm} \Longleftrightarrow &\min_{\substack{\overline{\blambda}_i \geq 0}} & ~~ \overline{\blambda}_i^T\bd\\ 
\mathrm{s.t.} &~~\bD \bmu \leq \bd &\mathrm{s.t.} & ~~ \bD^T \overline{\blambda}_i = \ba_i 
\end{align*} 
\textbf{Dual problem of \eqref{sub2}:}
\begin{align*}
\max_{\substack\bmu}&~~-\ba_i^T \bmu \hspace{1.2cm} \Longleftrightarrow &\min_{\substack{\underline{\blambda}_i \geq 0}} &~~\underline{\blambda}_i^T\bd\\ 
\mathrm{s.t.} &~~\bD \bmu \leq \bd &\mathrm{s.t.} & ~~ \bD^T \underline{\blambda}_i = -\ba_i 
\end{align*} 
\noindent for all $i = 1 ... n$. Finally, the resulting robust counterpart can be written as follows:
\begin{subequations} 
\begin{align*} 
 \textrm{(P2)} & \min_{\substack{\bG, \epsilon, t, \blambda}}  \hspace{.2cm} t - \gamma \epsilon \\
&\mathrm{subject\,to} \nonumber \\
&  \hspace{1cm} \overline{\blambda}_i^T \bd \leq t, \forall i = 1 ... n \\
&  \hspace{1cm} \underline{\blambda}_i^T \bd \leq t, \forall i = 1 ... n \\
&  \hspace{1cm} \bD^T \overline{\blambda}_i = \ba_i (\bG), \forall i = 1 ... n \\
&  \hspace{1cm} \bD^T \underline{\blambda}_i = -\ba_i (\bG), \forall i = 1 ... n \\
&  \hspace{1cm} \underline{\blambda}_i, \overline{\blambda}_i \geq 0, \forall i = 1 ... n \\
&  \hspace{1cm} \eqref{eqn:inequal_constr_mod},\eqref{eqn:epsilon_constr}, \eqref{eqn:pos_constr}
 \end{align*}
\end{subequations}
\noindent and $\blambda = [\overline{\blambda}_1^T, \underline{\blambda}_1^T, ... \overline{\blambda}_n^T, \underline{\blambda}_n^T]^T$. Recalling that $\ba_i (\bG)$ is a linear function of the elements of $\bG_p$ and $\bG_q$, it can be seen that the resulting robust counterpart (P2) is convex. A summary of the proposed approach  is illustrated in Figure \ref{Fig:flowchart}.

\begin{figure*}
\begin{center}
  \includegraphics[height=7cm]{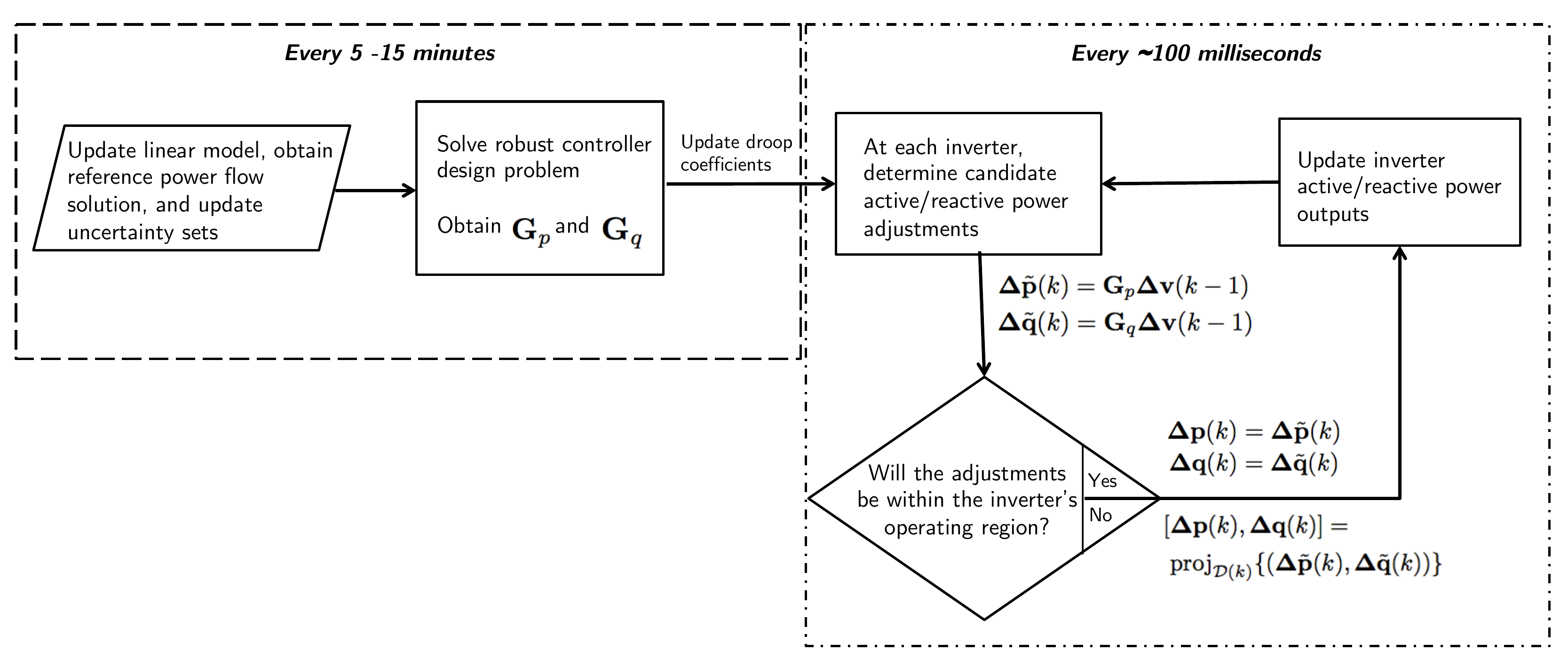} 
  \caption{Overall control strategy.}\label{Fig:flowchart}
  \end{center}
\end{figure*}

\section{Variants} \label{sec:variants}

\subsection{Participation Factors} \label{sec:fairness}

The effectiveness of droop control depends on the location of the inverter in the network. For example, in areas of the feeder with a high X/R ratio, Volt/VAR control can prove to be more effective \cite{EPRI_VAR}. However, due to this location dependency, the optimization problem considered in (P2) could, for example, lead to a situation where particular inverters are forced to participate more often or at a higher participation level than other inverters. In addition, if each inverter is voluntarily participating and being compensated for its contribution to voltage support, certain consumers may wish to penalize contribution of active power more than reactive power and have their own individual objectives, or choose not to participate at all during certain times of the day. Thus, in this subsection, we formulate an objective that allows for the Volt/VAR and Volt/Watt coefficients to be penalized differently at each individual inverter. Consider the following objective:
\begin{equation}
f(\be, \bG, \epsilon) = ||\be||_\infty - \gamma\epsilon + \bG_p^T \bM_p \bG_p + \bG_q^T \bM_q \bG_q \label{eq:fairness}
\end{equation}
\noindent where matrices $\bM_p$ and $\bM_q$ are diagonal and positive semidefinite weighting matrices that penalizes the contribution of active reactive power, respectively, from each inverter. A Volt/VAR-only control can be  obtained by either penalizing active power contribution with a large entries in $\bM_p$, or adding the constraint $\bG_p = 0$.

\subsection{Enabling Selection of Droop Locations} \label{sec:sparsity}

Communication limitations, planning considerations, and other motivating factors could influence the number of DERs that are installed in a certain area of the grid, or that are actively performing droop control within any given time interval. To consider this objective, the sparsity of the matrices $\bG_p$ and $\bG_q$ may be of interest. This can be achieved by minimizing the cardinality of the diagonals of these matrices; however, the cardinality function yields a combinatorial optimization formulation which may result in an intractable optimization problem. An alternative is to use a convex relaxation of the cardinality function, the $\ell_1$ norm \cite{Tibshirani94}, where $\|\bx\|_1 = \sum_{i=1}^{N}|x_i|$. Thus, the objective function in this case, simultaneously considering minimizing voltage deviations and sparsity, is the following:
\begin{align}
f(\be, \bG, \epsilon) = ||\be||_\infty - \gamma\epsilon &+ \eta_p || \diag(\bG_p) ||_1 \nonumber \\ 
&+ \eta_q || \diag(\bG_q) ||_1 \label{eq:sparsity}
\end{align}

\noindent where the $\diag(\cdot)$ operator takes the on-diagonal elements of a $n \times n$ matrix and creates a $n \times 1$ vector composed of these elements. The weighting parameters $\eta_p$ and $\eta_q$ can be individually tuned to achieve the desired level of sparsity for both $\bG_p$ and $\bG_q$ (the bigger $\eta_p$ and $\eta_q$, the more sparse these matrices will be). 

\section{Numerical and Sensitivity Analysis}
\label{sec:sims}

In this section, the modified IEEE 37-node test case will be discussed, and simulation results for the objectives considered in \eqref{eq:dev_obj}, \eqref{eq:fairness}, and \eqref{eq:sparsity} under the robust framework are shown. 

\subsection{IEEE Test Case}


The IEEE 37 node test system \cite{IEEE37} was used for the simulations, with 21 PV systems located at nodes 4, 7, 9, 10, 11, 13, 16, 17, 20, 22, 23, 26, 28, 29, 30, 31, 32, 33, 34, 35, and 36. For this experiment, a balanced single-phase equivalent of the test system is utilized; however, Section \ref{sec:sensitive} provides numerical results for the three-phase unbalanced case. One-second solar irradiance and load data taken from distribution feeders near Sacramento, CA, during a clear sky day on August 1, 2012~\cite{Bank13}, was used as the PV/Load inputs to the controller and are seen in Figure \ref{Fig:totLoadPV}. The stability margin parameter $\epsilon_0$ was set to $1^{-3}$, and $\gamma = 0.01$. After the optimal controller settings were determined using the linearlized power flow model, the deployed controller settings were simulated using the actual nonlinear AC power flows in MATPOWER \cite{MATPOWER}. The uncertainty set for the expected value of the real and reactive power fluctuations, $\cU$, was taken to be an interval with bounds on the maximum and minimum forecasted value for the power at each node over the upcoming control period.  

\begin{figure}[t]
\vspace{-2.6cm}
\hspace*{-.5cm} \includegraphics[width=10cm ]{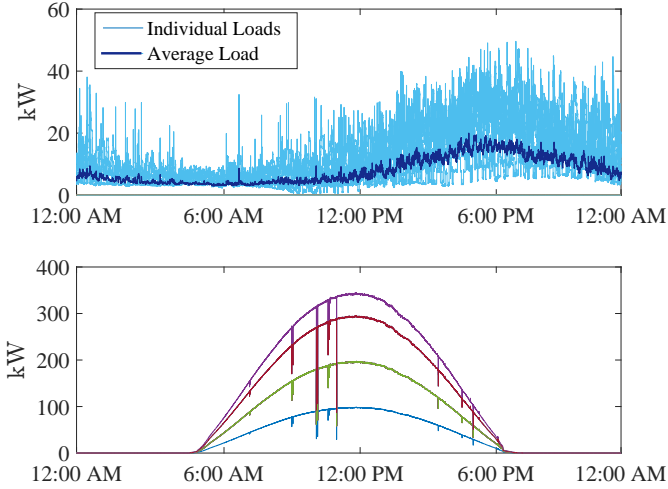}
\vspace{-3cm}
\caption{One-second data for the active power load at each node (top) and available solar generation at each inverter (bottom).}
\label{Fig:totLoadPV}
\vspace{-.2cm}
\end{figure}

\subsection{Locational Dependence of Droop Coefficients}

As will be demonstrated in the following, the optimal solution for the droop controllers is heavily location dependent. The following simulations were performed by choosing an objective that minimizes both voltage deviations and active power contribution (objective \eqref{eq:fairness} with $\bM_q = 0$ and $\bM_p = c\cdot \bI$; i.e., each inverter has equal penalty for Volt/Watt coefficients). The heatmap in Figure \ref{Fig:heatmap} illustrates the average magnitude of the desired droop settings for both Volt/VAR (top) and Volt/Watt (bottom) over four 15-minute control periods (11:00 AM - 12:00 PM). The higher magnitude of coefficients and thus increased voltage control towards the leaves of the feeder is consistent with previous research which has also found that voltage control can be most impactful when DERs are located near the end of distribution feeders \cite{Ranamuka14}.

In Figure \ref{Fig:PQSlopes_1hr}, the Volt/VAR and Volt/Watt coefficients are plotted for each inverter and each 15-minute control period. As the time approaches noon (i.e. as solar irradiance increases), the impact of active power control on mitigating voltage issues increases, as seen by the increase in Volt/Watt coefficients. Despite the penalty term in the objective on Volt/Watt coefficients and no penalty on Volt/VAR coefficients, active power control is still useful for voltage control in distribution networks due to the highly resistive lines and low X/R ratio \cite{Morren05}.

\begin{figure}[t]
\begin{center}
\vspace{-1cm}
\hspace{-3mm}\includegraphics[width=8cm ]{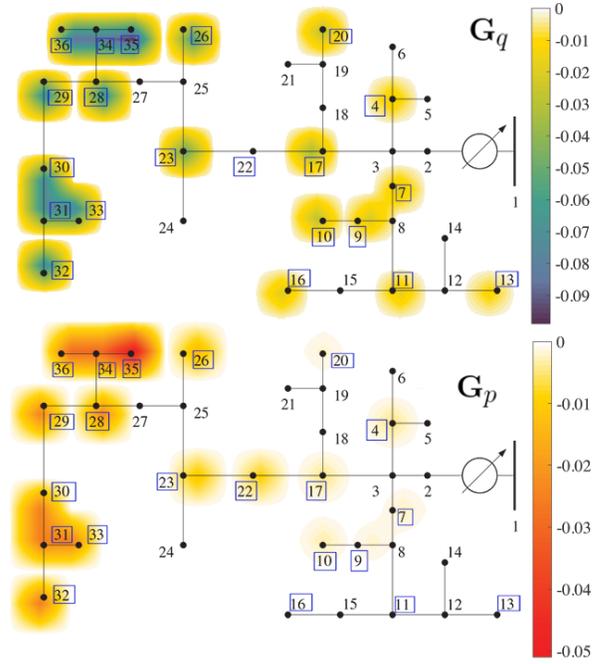}
\end{center}
\vspace{-1cm}
\caption{Heatmap of the average calculated droop coefficient at each inverter for Volt/VAR (top) and Volt/Watt (bottom) controllers over the course of 11:00 AM - 12:00 PM when active power contribution is penalized. Inverters, denoted with a rectangle around the node number, near the end of the feeder are expected to have a larger impact on voltage control.}
\label{Fig:heatmap}
\vspace{-.2cm}
\end{figure}

\begin{figure}[t]
\begin{center}
\hspace{-4mm}\includegraphics[width=9cm ]{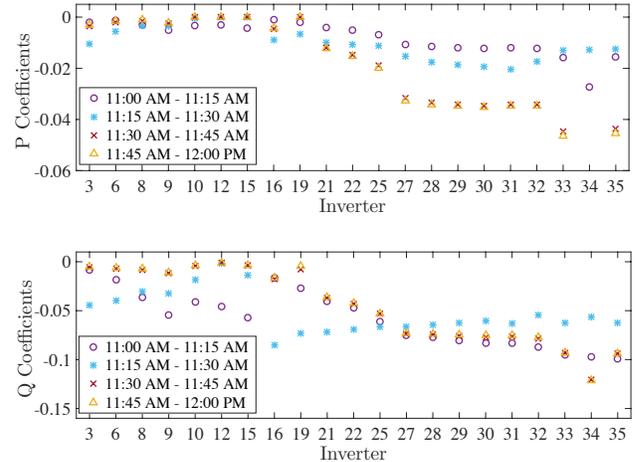}
\end{center}
\vspace{-.3cm}
\caption{Volt/VAR and Volt/Watt coefficients across all 21 inverters calculated every 15 minutes for a one hour period of 11:00 AM - 12:00 PM for the objective of minimizing voltage deviations away from 1 pu and Volt/Watt droop coefficients. As solar irradiance increases over time, active power has a more significant impact on voltage control, and thus the Volt/Watt coefficients become steeper.}
\label{Fig:PQSlopes_1hr}
\vspace{-.2cm}
\end{figure}

\subsection{Comparison}
To illustrate the benefits of the proposed methodology, a comparison with existing approaches to set the droop coefficients is provided next. We start with the case where the stability criterion is violated by increasing the value of the droop coefficients; this corresponds to the case where droop coefficients are determined in a network-agnostic way without system-level stability considerations. In the top subfigure in Figure \ref{Fig:5min_osc}, each droop coefficient was made steeper by -0.075. This overly aggressive control behavior results in voltage oscillations violating the upper $1.05$ pu bound, as seen in the figure. This motivates the use of explicitly including a constraint on stability in the optimization problem, rather than designing the controller according to heuristics. In addition to the potential of voltage oscillations, controllers whose settings are not updated over time may not be able to cope with the changing power and voltage fluctuations. We then consider a comparison with the Volt/VAR control settings specified in the IEEE 1547 guidelines \cite{IEEE1547}; see Figure \ref{Fig:V_1hr}. In comparison with the droop coefficients chosen via the Volt/VAR/Watt optimization problem, using the IEEE standard may result in undesirable voltage behavior, in this case violating the upper 1.05 pu bound. Lastly, there are some methods in the literature that design Volt/VAR droop coefficients based on the sensitivities at each node of reactive power to a change in voltage \cite{Jahangiri,Lee17}. However, designing the coefficients based on this heuristic offers no optimality guarantees; and as seen in Figure \ref{Fig:V_comp}, these coefficients can stabilize voltage deviations but result in undesirable voltage magnitudes (top subfigure). In the bottom subfigure of Figure \ref{Fig:V_comp}, only Volt/VAR control was implemented in the proposed framework to offer a fair comparison, and the droop coefficients were optimized to minimize voltage deviations. 

\begin{figure}[t]
\vspace{-2cm}
\begin{center}
\hspace{-4mm}\includegraphics[width=9cm ]{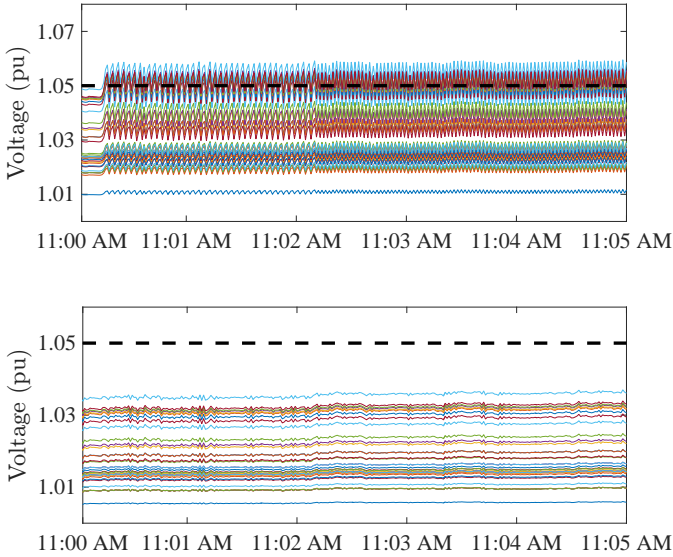}
\end{center}
\vspace{-2.2cm}
\caption{Voltage profiles for a five minute period with the droop coefficients decreased by -0.075 (top) and the proposed Volt/VAR/Watt droop control (bottom). Voltages oscillations occur when the coefficients are made more aggressive.}
\label{Fig:5min_osc}
\vspace{-.2cm}
\end{figure}
\begin{figure}[t!]
\vspace{-1.7cm}
\begin{center}
\hspace{-4mm}\includegraphics[width=8.5cm ]{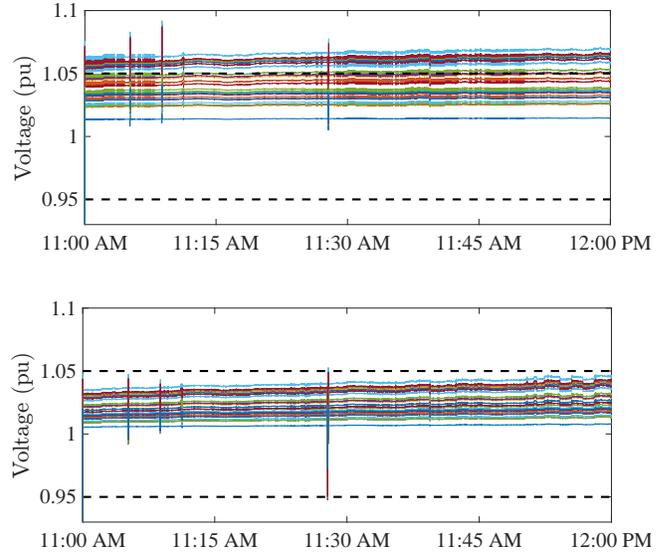}
\end{center}
\vspace{-1.8cm}
\caption{Voltages over an hour with the IEEE 1547 Volt/VAR standard (top) and the proposed Volt/VAR/Watt control (bottom). Voltages are between bounds with the optimized coefficients, whereas standard control results in overvoltages.}
\label{Fig:V_1hr}
\vspace{-.2cm}
\end{figure}

\begin{figure}[t!]
\vspace{-1.7cm}
\begin{center}
\hspace{-4mm}\includegraphics[width=8.5cm ]{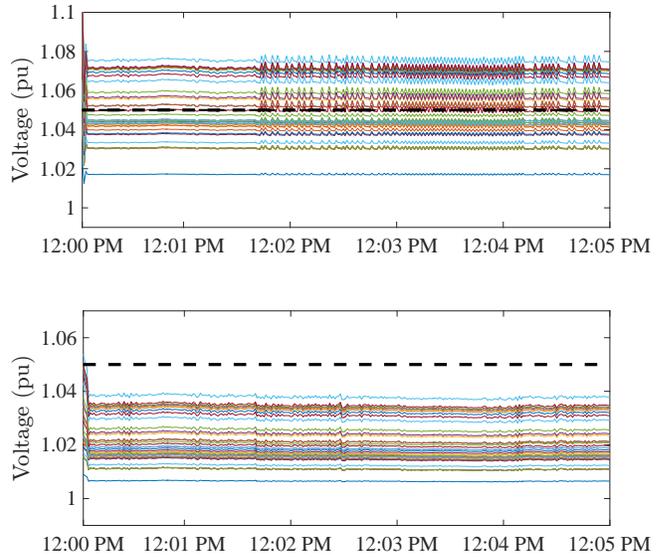}
\end{center}
\vspace{-1.8cm}
\caption{Voltage profiles obtained by calculating droop coefficients from a sensitivity matrix (top) and by using the proposed framework (bottom). Despite stabilizing the reactive power output of each inverter, calculating droop coefficients from voltage/reactive power sensitivities may sacrifice optimality and results in overvoltage conditions.}
\label{Fig:V_comp}
\vspace{-.2cm}
\end{figure}

\subsection{Controller Placement}

When planning for DER installation or when operating in a system constrained by communication limitations, there may be situations when the number of inverters participating in voltage support may be restricted. This objective, formulated in \eqref{eq:sparsity}, was used to optimize droop coefficients for 11:00 AM - 11:15 AM. The weighting parameters $\eta_p$ and $\eta_q$ were varied and the resulting coefficients from each of the cases are tabulated in Table \ref{tab:sparse}. In the first two columns where $\eta_p = \eta_q = 0$, the control matrices are full, and droop control is performed at every inverter. As expected, as the weighting terms increase, locations near the leaves of the feeder are selected as the most optimal for placement of the controllers. In the last column of the table, only one location is chosen to provide Volt/VAR support; however, it is worth noticing that the magnitude of the coefficient in this location is much greater than the individual coefficients when multiple inverters are participating. This is so that the impact of voltage control can still be high without the costly requirement of having multiple controllers. Overall, the location-dependence of the droop control highlights the value of droop control near the end of this particular feeder. When a limited number of droop controllers are available, the algorithm selects the most sensitive areas of the grid to provide the highest level of voltage regulation.

\begin{table}[]
\centering
\caption{Resulting droop coefficients when the number of controllers is penalized. }
\label{tab:sparse}
\begin{tabular}{|l|llllll}
\hline
\textbf{Node} & \multicolumn{2}{l|}{$\eta_p =\eta_q = 0$}  & \multicolumn{2}{l|}{$\eta_p =\eta_q = 0.001$}  & \multicolumn{2}{l|}{$\eta_p =\eta_q = 0.01$}   \\ \hline
\textbf{}     & \multicolumn{1}{l|}{$\bG_p$} & \multicolumn{1}{l|}{$\bG_q$} & \multicolumn{1}{l|}{$\bG_p$} & \multicolumn{1}{l|}{$\bG_q$} & \multicolumn{1}{l|}{$\bG_p$} & \multicolumn{1}{l|}{$\bG_q$} \\ \hline
\textbf{4}    & -0.002                             & -0.009                             & 0                                  & 0                                  & 0                                  & 0                                  \\ \cline{1-1}
\textbf{7}    & -0.001                             & -0.019                             & 0                                  & 0                                  & 0                                  & 0                                  \\ \cline{1-1}
\textbf{9}    & -0.003                             & -0.037                             & -0.001                            & 0                                  & 0                                  & 0                                  \\ \cline{1-1}
\textbf{10}   & -0.005                             & -0.055                             & -0.003                            & -0.002                             & 0                                  & 0                                  \\ \cline{1-1}
\textbf{11}   & -0.003                             & -0.041                             & -0.004                            & -0.010                              & 0                                  & 0                                  \\ \cline{1-1}
\textbf{13}   & -0.003                             & -0.046                             & -0.013                            & -0.044                             & 0                                  & 0                                  \\ \cline{1-1}
\textbf{16}   & -0.004                             & -0.057                             & -0.005                             & -0.024                             & 0                                  & 0                                  \\ \cline{1-1}
\textbf{17}   & -0.001                             & -0.017                             & 0                            & 0                                  & 0                                  & 0                                  \\ \cline{1-1}
\textbf{20}   & -0.002                             & -0.027                             & -0.001                            & 0                                  & 0                                  & 0                                  \\ \cline{1-1}
\textbf{22}   & -0.004                             & -0.041                             & -0.005                            & -0.025                             & 0                                  & 0                                  \\ \cline{1-1}
\textbf{23}   & -0.005                             & -0.047                             & -0.006                            & -0.035                             & 0                                  & 0                                  \\ \cline{1-1}
\textbf{26}   & -0.007                             & -0.061                             & -0.009                            & -0.051                             & 0                                  & 0                                  \\ \cline{1-1}
\textbf{28}   & -0.011                             & -0.075                             & -0.014                            & -0.073                             & -0.002                             & 0                                  \\ \cline{1-1}
\textbf{29}   & -0.012                             & -0.077                             & -0.015                            & -0.074                             & -0.004                             & 0                                  \\ \cline{1-1}
\textbf{30}   & -0.012                             & -0.081                             & -0.015                             & -0.075                             & -0.005                             & 0                                  \\ \cline{1-1}
\textbf{31}   & -0.012                             & -0.083                             & -0.015                            & -0.077                             & -0.006                             & 0                                  \\ \cline{1-1}
\textbf{32}   & -0.012                             & -0.083                             & -0.015                            & -0.077                             & -0.006                             & 0                                  \\ \cline{1-1}
\textbf{33}   & -0.012                             & -0.087                             & -0.015                            & -0.080                              & -0.006                             & 0                                  \\ \cline{1-1}
\textbf{34}   & -0.016                             & -0.095                             & -0.020                            & -0.102                             & -0.020                              & 0                                  \\ \cline{1-1}
\textbf{35}   & -0.027                             & -0.097                             & -0.037                            & -0.132                             & -0.056                             & -0.311                             \\ \cline{1-1}
\textbf{36}   & -0.016                             & -0.099                             & -0.019                            & -0.104                             & -0.019                             & 0                                  \\ \cline{1-1}
\end{tabular}
\end{table}

\subsection{Computational Burden and Neuman Approximation}\label{sec:sensitive}
In this section, we provide numerical indications regarding the growth of the computational burden with respect to the problem size as well as the sensitivity of the solution to changes in penalty terms. These simulations were performed on a single Macbook Pro laptop with a 3.1 GHz Intel Core i7 and 16 GB of RAM. The problem was solved using MATLAB with the publicly available SDPT3 solver through the CVX interface.

\subsubsection{Test Cases}
Four different settings for the objective function are considered when designing the droop coefficients:

\begin{itemize}
\item \textbf{Case I:} Coefficients for both active and reactive powers are computed (i.e., Volt/VAr/Watt control);
\item textbf{Case II:} Penalization of active power contributions ($\bM_p \gg 1$); 
\item \textbf{Case III:} Penalization reactive power contribution ($\bM_q \gg 1$);
\item \textbf{Case IV:} Number of controllers penalized ($\eta_p = \eta_q = 0.001$)
\end{itemize}

\subsubsection{Sensitivity Analysis}

Since the proposed methodology utilizes an approximation of \eqref{eqn:equal_constr} in order to obtain a convex optimization problem, numerical experiments are performed next to evaluate the approximation error. The four cases stated in the previous subsection were solved and an optimal $\bG$ was obtained for each case. In Figure \ref{Fig:reaL_err}, a parameter $\alpha$ was varied from 0 to 1 and the relative approximation error between $(\bI - \bH\bG')^{-1}$ and $(\bI + \bH\bG')$ was assessed, where $\bG' = \alpha \cdot \bG$. It can be seen from Figure \ref{Fig:reaL_err} that the approximation error does not exceed $10\%$.

\begin{figure}[t!]
\vspace{-3.5cm}
\begin{center}
\hspace{-4mm}\includegraphics[width=8.5cm ]{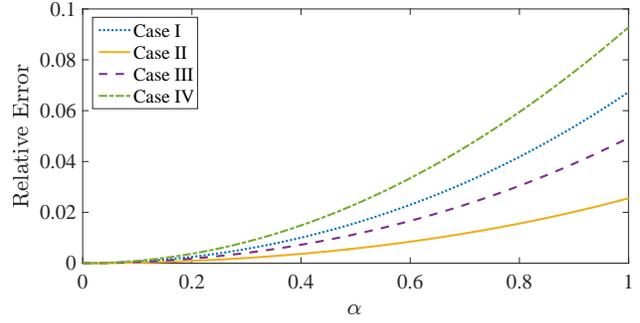}
\end{center}
\vspace{-3.5cm}
\caption{Sensitivity analysis for the truncated Neuman series approximation for the considered four cases.}
\label{Fig:reaL_err}
\vspace{-.2cm}
\end{figure}

\subsubsection{Implementation in multi-phase systems} 

We next consider the full three-phase version of the IEEE 37-node test system was implemented to further assesses how the computation time  scales with the problem size. It is assumed that every node in the system consisted of three phases, and that each phase may have inverters. The three-phase linearization of the power flow equations in \cite{multiphaseArxiv} is used. The average computational time is measured for each case over five runs. As seen in Table \ref{tab:comp_time}, despite the decision matrices dimensions increasing threefold, the computational time is within seconds.

\begin{table}[]
\centering
\caption{Computational time required to solve the optimization problem.}
\label{tab:comp_time}
\begin{tabular}{|l|ll}
\hline
Computational Time (s) & \multicolumn{1}{l|}{Single-Phase} & \multicolumn{1}{l|}{Three-Phase} \\ \hline
Case I                 & 1.05                              & 1.98                             \\ \cline{1-1}
Case II                & 1.15                              & 2.51                             \\ \cline{1-1}
Case III               & 0.92                              & 2.34                             \\ \cline{1-1}
Case IV                & 0.49                              & 1.51                             \\ \cline{1-1}
\end{tabular}
\end{table}

\begin{figure}[t!]
\vspace{-3.5cm}
\begin{center}
\hspace{-4mm}\includegraphics[width=8.5cm ]{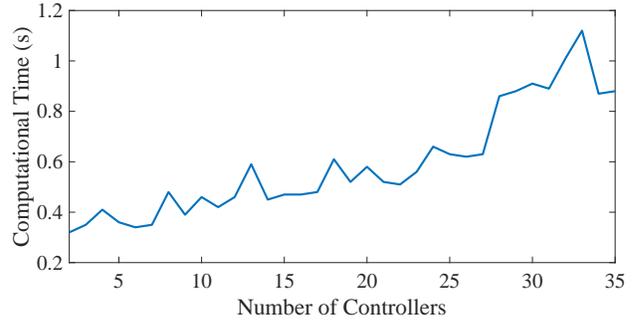}
\end{center}
\vspace{-3.5cm}
\caption{Sensitivity analysis demonstrating how the total simulation time increases as the number of controllers increases. The required amount of computational time as the number of controllers increases is well under the 5-15 minute time window available to solve the optimization problem.}
\label{Fig:comp_time}
\vspace{-.2cm}
\end{figure}

Results regarding the computational time for different number of inverters are provided next. Figure \ref{Fig:comp_time} demonstrates how the computational time increases as the number of controllers increases for Case I, measured using \texttt{cputime} in CVX. As expected, the computational burden increases with solution space size. 

\section{Conclusion}
\label{sec:conclusion}

The paper addressed the design of proportional control strategies for DERs for voltage regulation purposes.  The design of the coefficients of the controllers leveraged suitable linear approximations of the AC power-flow equations and is robust to uncertainty in the forecasted non-controllable loads/power injections. Stability of the proposed local controllers when deployed in the actual network (i.e., considering nonlinear AC power-flow equations in the analysis) was analytically established. 

The simulation results highlighted that the proposed controllers exhibit superior performance compared to the recommended IEEE 1547.8 Volt/VAR settings in terms of stability and voltage regulation capabilities, as well as compared to methods in the literature which use sensitivity matrices to design the droop coefficients. Particularly, if the droop coefficients are not tuned properly or set using rule-of-thumb guidelines, voltage oscillations can occur due to fast timescale fluctuations in load and solar irradiance, or under/over voltage conditions may be encountered. 

A sensitivity analysis was performed regarding the approximation error induced by using the truncated Neuman series, and how the computational burden changes with respect to the number of controllers. The overall framework provides a light-weight, yet powerful, method of updating existing Volt/VAR droop coefficients in advanced inverters to achieve a variety of objectives while ensuring voltage stability under uncertainty.

\appendix[]

\begin{proof}[Proof of Theorem \ref{thm:stab}]
Let $\hat{\bz}(k) :=  \bz(k) + \proj_{\cD(k)} (\bG \bDelta \bv(k-1))$. We have that
\begin{align} \label{eqn:deltaV_ineq}
 &\|\bDelta \bv(k)\|_2 \leq \left \| F_L(\hat{\bz}(k))  - F_L(\bar{\bz}) \right \|_2 + \left \| F_L(\hat{\bz}(k))  - F(\hat{\bz}(k)) \right \|_2 \nonumber \\
&\quad \leq \left \| \bH \bDelta \bz_{nc}(k) + \bH\proj_{\cD(k)} (\bG \bDelta \bv(k-1)) \right \|_2 + \delta \nonumber\\
&\quad \leq \|\bH\|_2 C + \|\bH\|_2 \| \proj_{\cD(k)} (\bG \bDelta \bv(k-1)) \|_2 + \delta \nonumber\\
&\quad \leq \|\bH\|_2 C + \|\bH\|_2 \|\bG \bDelta \bv(k-1) \|_2 + \delta 
\end{align}
where the second inequality follows by Assumption \ref{asm:lin_mod} and the definition of the linear model \eqref{eq:approximate}, the third inequality holds by the hypothesis that $\|\bDelta \bz_{nc}(k)\|_2 \leq C$, and in the last inequality the non-expansive property of the projection operator was used; in particular, as ${\bf 0} \in \cD(k)$ for all $k$, we have that $\| \proj_{\cD(k)} (\bx) \|_2 \leq \| \bx \|_2$ for all $k$ and any $\bx$.

We next proceed to obtain a bound on $\|\bG \bDelta \bv(k-1) \|_2$. Similarly to the derivation in \eqref{eqn:deltaV_ineq}, it holds that
\begin{align} \label{eqn:deltaGV_ineq_rec}
&\|\bG \bDelta \bv(k) \|_2 \nonumber \\
&\, \leq \left \| \bG \bH \bDelta \bz_{nc}(k) + \bG \bH\proj_{\cD(k)} (\bG \bDelta \bv(k-1)) \right \|_2 + \| \bG \|_2 \delta \nonumber \\
&\, \leq \|\bG \bH\|_2 C + \|\bG \bH\|_2 \| \proj_{\cD(k)} (\bG \bDelta \bv(k-1)) \|_2 + \|\bG \|_2\delta \nonumber\\
&\, \leq rC + r\|\bG \bDelta \bv(k-1) \|_2 + \|\bG \|_2\delta.
\end{align}
By applying \eqref{eqn:deltaGV_ineq_rec} recursively, we obtain 
\begin{align}\label{eqn:deltaGV_ineq}
\|\bG \bDelta \bv(k) \|_2 &\leq (rC + \|\bG \|_2\delta) \sum_{i = 0}^{k-1}r^i  + r^k \|\bG \bDelta \bv(0) \|_2 \nonumber \\
& = (rC + \|\bG \|_2\delta) \frac{1 - r^k}{1 - r} + r^k \|\bG \bDelta \bv(0) \|_2.
\end{align}
Now, plugging \eqref{eqn:deltaGV_ineq} in \eqref{eqn:deltaV_ineq} yields
\begin{align*}
&\|\bDelta \bv(k)\|_2 \leq \|\bH\|_2 C \\
&\, + \|\bH\|_2 \left( (rC + \|\bG \|_2\delta) \frac{1 - r^{k-1}}{1 - r} + r^{k-1} \|\bG \bDelta \bv(0) \|_2 \right) + \delta.
\end{align*}
The proof is then completed by taking $\limsup$ and rearranging.
\end{proof}

\bibliographystyle{IEEEtran}
\bibliography{biblio}

\newpage
\begin{IEEEbiography}[{\includegraphics[width=1in]{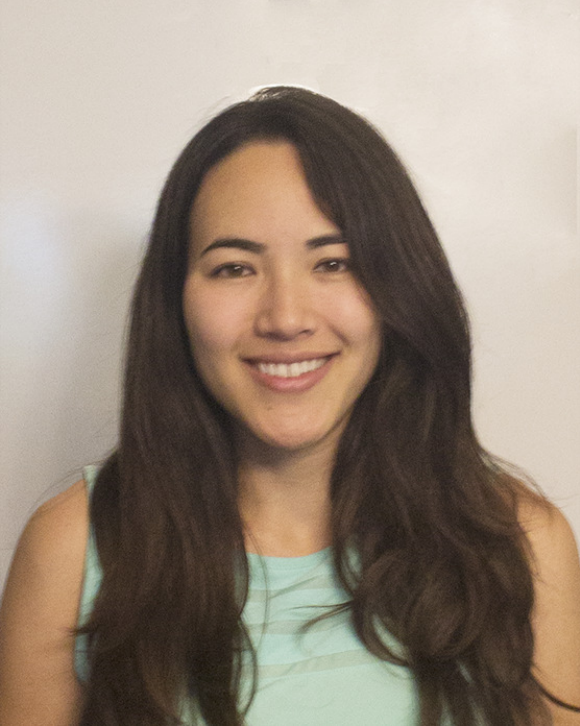}}]{Kyri Baker} (S'08, M'15)
received her B.S., M.S, and Ph.D. in Electrical and Computer Engineering at Carnegie Mellon University in
2009, 2010, and 2014, respectively. Since Fall 2017, she has been an Assistant Professor at the University of Colorado, Boulder, in the Department of Civil, Environmental, and Architectural Engineering, with a courtesy appointment in the Department of Electrical, Computer, and Energy Engineering. Previously, she was a Research Engineer at the National Renewable Energy Laboratory in Golden, CO. Her research interests include power system optimization and planning, building-to-grid integration, smart grid technologies, and renewable energy.
\end{IEEEbiography}

\begin{IEEEbiography}[{\includegraphics[width=1in]{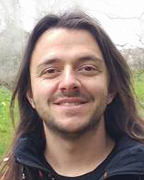}}]{Andrey Bernstein} (M'15)
received the B.Sc. and M.Sc. degrees in Electrical Engineering from the Technion - Israel Institute of Technology in 2002 and 2007 respectively, both summa cum laude. He received the Ph.D. degree in Electrical Engineering from the Technion in 2013. Between 2010 and 2011, he was a visiting researcher at Columbia University. During 2011-2012, he was a visiting Assistant Professor at the Stony Brook University. From 2013 to 2016, he was a postdoctoral researcher at the Laboratory for Communications and Applications of Ecole Polytechnique Federale de Lausanne (EPFL), Switzerland. Since October 2016 he has been a Senior Scientist at the National Renewable Energy Laboratory, Golden, CO, USA.
\end{IEEEbiography}

\begin{IEEEbiography}[{\includegraphics[width=1in]{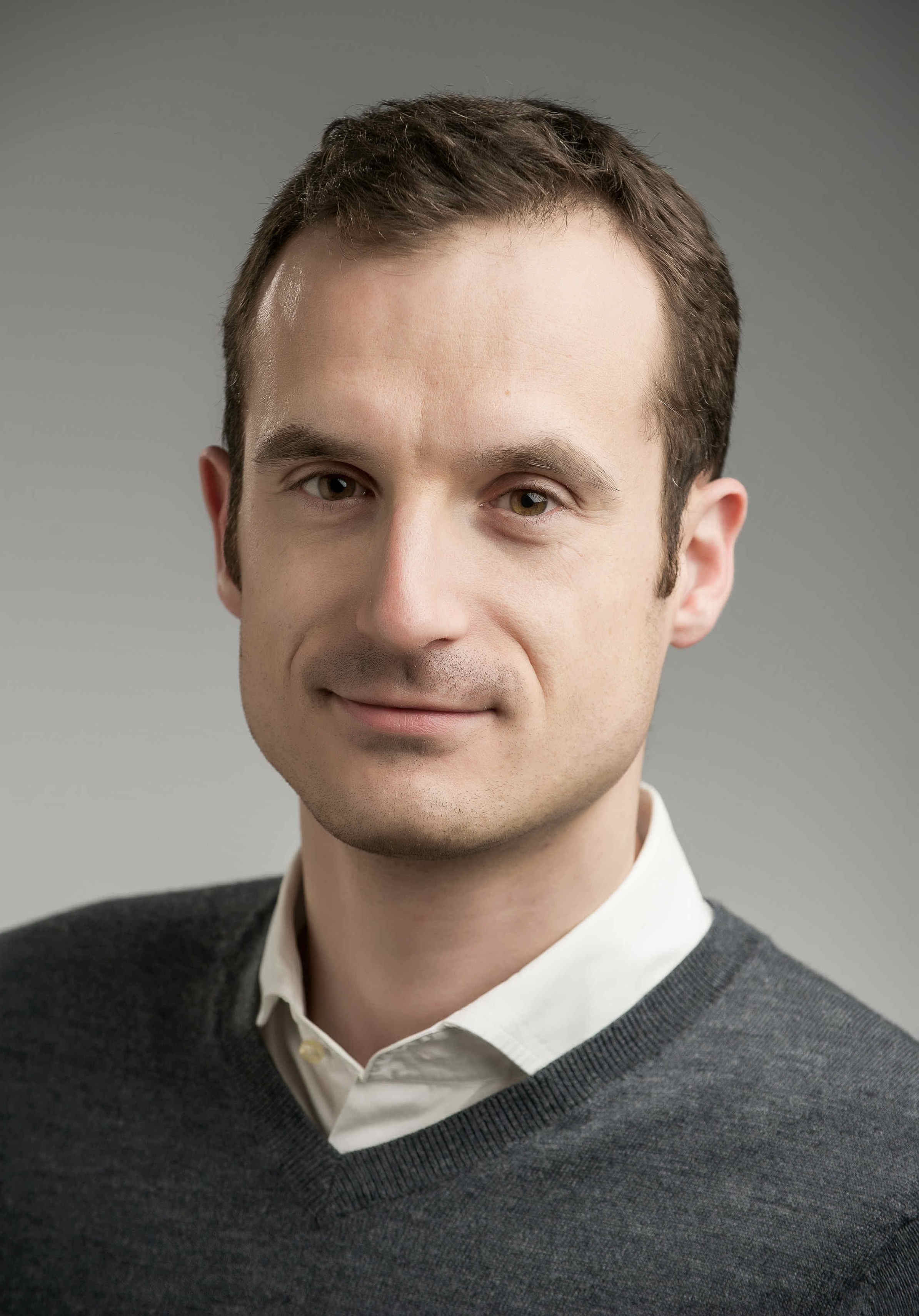}}]{Emiliano Dall'Anese } (S'08, M'11)
received the Laurea Triennale (B.Sc Degree) and the Laurea Specialistica (M.Sc Degree) in Telecommunications Engineering from the University of Padova, Italy, in 2005 and 2007, respectively, and the Ph.D. In Information Engineering from the Department of Information Engineering, University of Padova, Italy, in 2011. From January 2009 to September 2010, he was a visiting scholar at the Department of Electrical and Computer Engineering, University of Minnesota, USA. From January 2011 to November 2014 he was a Postdoctoral Associate at the Department of Electrical and Computer Engineering and Digital Technology Center of the University of Minnesota, USA. Since December 2014 he has been a Senior Engineer at the National Renewable Energy Laboratory, Golden, CO, USA.

His research interests lie in the areas of optimization and signal processing, with application to power systems and communication networks. Current efforts focus on distributed and online optimization of power distribution systems with distributed (renewable) energy resources, and statistical inference for grid data analytics.
\end{IEEEbiography}

\begin{IEEEbiography}[{\includegraphics[width=1in]{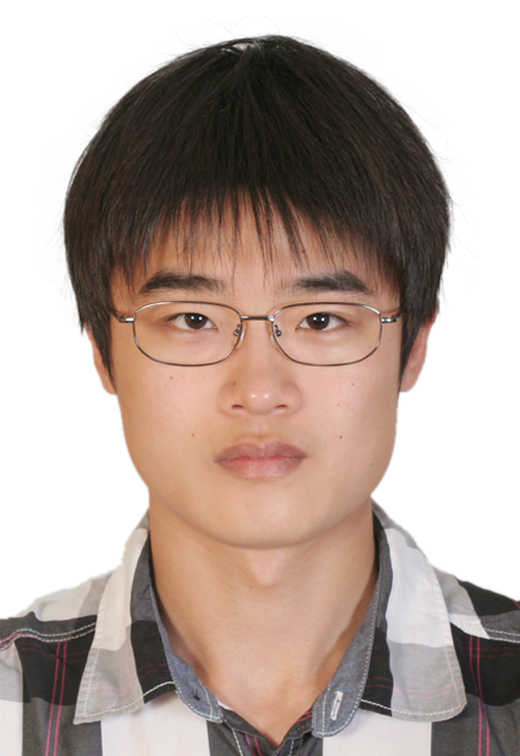}}]{Changhong Zhao} (S'12, M'15)
received the B. Eng. degree in Automation from Tsinghua University in 2010, and the PhD degree in Electrical Engineering from California Institute of Technology in 2016. He is currently a Research Engineer with the Power Systems Engineering Center at National Renewable Energy Laboratory, Golden, CO, USA. His research interests  include power system dynamics and stability, distributed control, and optimal power flow. He was a recipient of the Caltech Demetriades-Tsafka-Kokkalis PhD Thesis Prize and the Caltech Charles Wilts Prize for Doctoral Research.
\end{IEEEbiography}

\end{document}